\documentclass[12 pt]{amsart}

\usepackage[margin=2.3cm]{geometry} 
\usepackage[utf8]{inputenc}
\usepackage{amsmath,amssymb}
\usepackage[colorlinks=true,
linkcolor=blue,
citecolor=blue,
urlcolor=blue]{hyperref}

\newtheorem{theorem}{Theorem}

\theoremstyle{definition}

\theoremstyle{remark}

\newcommand{\N}{\mathbb{N}}

\begin{document}

\title[$P$-Contractive implies Picard  ]
{A Short Note on  \(P\)-Contractive  and
	 Picard Operators}

\author{Faruk Temur}
\address{Department of Mathematics\\
	Izmir Institute of Technology, Urla, Izmir, 35430,
	Turkey}
\email{faruktemur@iyte.edu.tr}
\date{August 18, 2026}

\maketitle

	\maketitle
	
	\begin{abstract}
	Altun and Hancer \cite{AltunHancer2019} proved that every continuous \(P\)-contractive self-map of a
		compact metric space is an almost Picard operator, and asked whether it must in
		fact be a Picard operator. We give an affirmative answer. The proof follows by combining the existing results with a final compactness argument.
	\end{abstract}
	
	\section{Introduction}
	
Let \((X,d)\)  be a metric space. 	A map \(T\colon X\to X\) is called
\emph{\(P\)-contractive} if
\begin{equation}\label{eq:Pcontractive}
	d(Tx,Ty)
	< d(x,y)+\bigl|d(x,Tx)-d(y,Ty)\bigr|
\end{equation}
for all distinct \(x,y\in X\). 	On the other hand, a  map \(T\colon X\to X\)   is called a \emph{Picard operator} if it has a unique
fixed point $z$,
 and
\[T^n x\rightarrow z, \quad \text{for every} \ \ x\in X.\]
 For a map \(T\colon X\to X\),  and a  given point
\(x_0\in X\), the sequence
\[
x_{n+1}=Tx_n,\qquad n\ge 0,
\]
is called the \emph{Picard iteration} or \emph{Picard orbit} generated by
\(x_0\).

In \cite{AltunHancer2019},	Altun and Hancer proved
	that a continuous \(P\)-contractive self-map of a compact metric space has a
	unique fixed point, and that every Picard orbit has a subsequence converging to
	that point; in their terminology, such a map is an almost Picard operator. They asked whether the full Picard conclusion holds.
	The same question was recorded as open by Başar and Altun
	\cite{BasarAltun2026}. We answer this question  affirmatively.

	\begin{theorem}\label{thm:main}
		Let \((X,d)\) be a nonempty compact metric space, and let
		\(T\colon X\to X\) be a continuous \(P\)-contractive mapping. Then \(T\) is a
		Picard operator.
	\end{theorem}

 The proof of the  theorem  has  3 steps, and fortunately first two of these already exist in the literature \cite{AltunDurmazOlgun2018,AltunHancer2019,Garnicki2019}. So all we need to do is to combine the first two, and build  the final step on that combination.

	\section{Proof of the Main result}
	In this section we prove our main theorem.

\begin{proof}[Proof of the Main Theorem] Our three steps are as follows.\\

{\bf  I. } The map $T$ has a unique fixed point. This is known by Theorem 2.14 \cite{AltunDurmazOlgun2018}.\\

{\bf II.} Every Picard orbit is asymptotically regular, that is $d(x_{n-1},x_{n})\rightarrow 0, \ n\in \N $ for any $x_0\in X$. This is shown in the proof of Theorem 2 in  \cite{AltunHancer2019}. As Lemma 1.3 of \cite{Garnicki2019} states, this implies that every limit point of a Picard orbit is a fixed point.  Since $T$ has a unique fixed point, any limit point of any Picard orbit is this point.\\

\textbf{III.} By first two steps we know that $T$ has a unique fixed point $z$, and any limit point of any Picard orbit is this point. It remains to show that any Picard orbit converges to this point. Pick an arbitrary $x_0\in X$ and consider its orbit $\{x_n\}.$
 If	this orbit did not converge to \(z\), there would exist \(\varepsilon>0\) and a
		subsequence \((x_{n_j})\) such that
		$
		d(x_{n_j},z)\ge\varepsilon$  for every $
		j\in \N.
		$
		But compactness of $X$ would provide a further subsequence converging to $z$, contradicting this. Hence any orbit converges to $z.$ This finishes the proof.

	\end{proof}

\end{document}